\documentclass[a4paper]{amsart}[12pt]

\usepackage[latin1]{inputenc} 
\usepackage[T1]{fontenc}      % Police contenant les caractères français
\usepackage{geometry}         % Définir les marges
\usepackage[english]{babel}  % Placez ici une liste de langues, la
\usepackage{amsmath}                            
\usepackage{amsfonts}
\usepackage{amsthm}                             
                             
\usepackage{amscd}

   \usepackage{amssymb}
   \usepackage{verbatim}
   \newif\ifpdf
   \ifpdf
     \usepackage[pdftex]{graphicx}
     \usepackage[pdftex]{hyperref}
   \else
     \usepackage{graphicx}
   \fi
   
\usepackage{dsfont}\let\mathbb\mathds
\usepackage[all]{xy}
   
   \numberwithin{equation}{section}       % Number formulas within sections
    \theoremstyle{plain}    
    \newtheorem{thm}{Theorem}[section]
    \numberwithin{equation}{section} %% Comment out for sequentially-numbered
    \numberwithin{figure}{section} %% Comment out for sequentially-numbered
    \theoremstyle{plain}
    \theoremstyle{plain}    
    \theoremstyle{plain}    
    \theoremstyle{plain}    
    \newtheorem{lem}[thm]{Lemma} %%Delete [thm] to re-start numbering
    \theoremstyle{remark}
    
    \theoremstyle{definition}
   
   \newtheorem*{thmM}{Main Theorem}
   \newtheorem*{thmA}{Theorem A}

   \theoremstyle{definition}
   \newtheorem{defi}[thm]{Definition}

\newcommand{\R}{\mathbb{R}}  
 
\newcommand{\N}{\mathbb{N}} 
\newcommand{\Z}{\mathbb{Z}} 
\newcommand{\C}{\mathbb{C}}

\newcommand{\teich}{\mathrm{Teich}}
\newcommand{\belf}{\mathrm{bel}(f)}
\newcommand{\Belf}{\mathrm{Bel}(f)}

\newtheorem{prop}{Proposition}   
\newtheorem{coro}{Corollary}

\newcommand{\dbar}{\overline{\partial}}
\newcommand{\rs}{\mathbb{P}^1}
\newcommand{\res}{\mathrm{Res}}
\newcommand{\ratd}{\mathrm{Rat}_d}

\newcommand{\limn}{\lim_{n \rightarrow \infty}}
\newcommand{\id}{\mathrm{Id}}
\newcommand{\critf}{\mathrm{Crit}(f)}
\newcommand{\aut}{\mathrm{aut}}
\newcommand{\of}{\mathcal{O}(f)}

\newcommand{\qc}{\mathrm{QC}}

\newcommand{\s}{\mathcal{S}}
\newcommand{\fatou}{\mathcal{F}}
\newcommand{\julia}{\mathcal{J}}
\newcommand{\modspace}{\mathrm{rat}_d}

\title{On the dynamical Teichmüller space}

\author{Matthieu Astorg}
\thanks{This work was partially supported by the ANR Lambda project.}
\address{Institut de Mathématiques de Toulouse \\
Université Paul Sabatier \\
118 route de Narbonne \\
31 062 Toulouse, France}
\email{mastorg@math-univ.toulouse.fr}

\keywords{Complex dynamics; Teichmüller theory; Quasiconformal analysis}

\begin{document}

\maketitle

\begin{abstract}
We prove that the Teichmüller space of a rational map immerses into the moduli
space of rational maps of the same degree, answering a question of McMullen and Sullivan. 
This is achieved through a new description of the tangent and cotangent space of 
the dynamical Teichmüller space.
\end{abstract}

\section{Notations}

The following notations will be used throughout the article : 
\begin{itemize}
\item $\s$ is a Riemann surface
\item $\rs$ is the Riemann sphere
\item $\Omega$ is a hyperbolic open subset of $\rs$
\item $f : \rs \rightarrow \rs$ is a rational map
\item If $\s$ is hyperbolic, $\rho_\s$ is the hyperbolic metric on $\s$
\end{itemize}

\section{Introduction}

Let us denote by  $\ratd$ the space of rational fractions of degree $d$, 
and by $\modspace$ its quotient under the action by conjugacy of the group
of Möbius transformations. For  $f \in \ratd$, we will denote by $\of$  the orbit of $f$ 
under the action of the group of Möbius transformations.

In order to study the geometry of the quasiconformal
conjugacy class of $f$ in $\ratd$ and $\modspace$, McMullen 
and Sullivan introduced in \cite{mcmullen1998quasiconformal} the dynamical 
Teichmüller space of a rational map $f$, as a dynamical analogue of the Teichmüller theory
of surfaces (see \cite{gardiner2000quasiconformal} and \cite{hubbard2006teichmuller} for 
an introduction to Teichmüller theory). McMullen and Sullivan constructed a natural complex
structure on the Teichmüller space of a rational map $f$ of degree $d$, making it into
a complex manifold of dimension at most $2d-2$. They also exhibited a holomorphic map of orbifolds
$\Psi : \teich(f) \rightarrow  \modspace$ whose image is exactly the quasiconformal 
conjugacy class of $f$ : thus one should think of 
the Teichmüller space of $f$ as a complex manifold parametrizing the conjugacy class of $f$.
In this context, a natural question arises concerning the parametrization $\Psi$ : 
is it an immersion ? This question was asked by McMullen and Sullivan in their introductory paper.
As it turns out, the answer is yes.
Adam Epstein has an unpublished proof of this result; in 
\cite{makienko2008remarks}, Makienko also gives a proof in the same spirit.
We present here a different approach, using more elementary tools : in particular, we won't need 
the explicit description of the Teichmüller space of $f$ given in \cite{mcmullen1988quasiconformal}, 
and we will give a new method for constructing the complex structure on $\teich(f)$ which does not rely 
on preexisting Teichmüller theory.  
A related question, also raised in \cite{mcmullen1998quasiconformal}, is to know whether or not 
the image of this map can accumulate on itself. In \cite{branner1992cubic}, Branner showed
that the answer is yes. \newline

Denote by $\belf$ the space of $L^\infty$ Beltrami differentials invariant under $f$
, and by $\Belf$ its unit ball. We shall use the term "Beltrami forms" for elements of 
$\Belf$, and "Beltrami differentials" for elements of $\belf$, which we will think of 
as the tangent space to $\Belf$. Given a quasiconformal homeomorphism $\phi$, we 
will denote by $K(\phi)$ its dilatation.

\begin{defi}
\begin{itemize}
\item Denote by $\qc(f)$ the group of quasiconformal homeomorphisms commuting with $f$.
\item Denote by $\qc_0(f)$ the normal subgroup of the elements $\phi \in \qc(f)$ such that 
there exists $K>1$ and an isotopy $\phi_t \in \qc(f)$ with $\phi_0 = \mathrm{Id}$, $\phi_1 = \phi$
and for all $t \in [0,1]$, $K(\phi_t) \leq K$. 
\item The modular group of $f$ is $\mathrm{Mod}(f) = \qc(f)/\qc_0(f)$.
\item The Teichmüller space of a rational map $f$ (which we will denote by $\teich(f)$)
is $\Belf$ quotiented by the right action of $\qc_0(f)$ by precomposition.
\end{itemize}

\end{defi}

Let $Z \subset \rs$ be a set of cardinal $3$.
There is a holomorphic map $\Psi^Z : \Belf \rightarrow \ratd$ defined by
$\Psi(\mu)=\phi_\mu^Z \circ f \circ ({\phi_\mu^Z})^{-1}$, where $\phi_\mu^Z$ is the
unique  solution of the Beltrami equation $\dbar \phi_\mu^Z = \mu \circ \partial \phi_\mu^Z$
fixing $Z$. It descends to a holomorphic map of orbifolds
$\Psi : \Belf \rightarrow \modspace$ independant from the choice of $Z$, and 
to maps $\Psi_T^Z : \teich(f) \rightarrow \ratd$ and $\Psi_T : \teich(f) \rightarrow \modspace$.

The unit ball $\Belf$ being an open subset of the Banach space $L^\infty$, it has a 
natural complex Banach manifold structure,
and there exists at most one complex structure on $\teich(f)$ making $\pi : \Belf \rightarrow \teich(f)$ into a split submersion. Using the results of \cite{mcmullen1988quasiconformal} on
the equivalence between several notions of isotopies (isotopies relative to the ideal 
boundary, relative to the topological boundary, uniformly quasiconformal isotopies)
McMullen and Sullivan constructed such a complex structure on $\teich(f)$ and showed that
$\teich(f)$ is isomorphic to the cartesian product of a polydisk and of Teichmüller spaces of 
some finite type Riemann surfaces associated to the dynamics of $f$.

Once  $\teich(f)$ is endowed with its complex structure, one can verify that $\Psi_T^Z$ and $\Psi_T$ are holomorphic maps between complex manifolds and orbifolds respectively, and McMullen 
and Sullivan asked whether those maps are immersions. Since $\modspace$ is not 
a manifold, we have to define what we mean by the statement that $\Psi_T$ is an immersion.

\begin{defi}
We will say that $\Psi_T$ is an immersion if the lift $\Psi_T^Z$ is an immersion
 whose image is transverse to
$\of$. If this is true for one choice of normalization set $Z$, then it holds
for all $Z$.
\end{defi}

It turns out that $\Psi_T$ is indeed an immersion,
 and Adam Epstein has an unpublished proof of this result.
The idea of his proof is a dual approach using quadratic differentials. 
The key ingredients are the deformation spaces introduced in  
 \cite{epstein2009transversality} 
and a result of Bers concerning the density of rational quadratic differentials
 (cf \cite{gardiner2000quasiconformal},
theorem 9 p.63).

The main result of this article is to give another proof of this result : 

\begin{thmM}
The map $\Psi_T : \teich(f) \rightarrow \modspace$ is an immersion.
\end{thmM}

Our proof uses a new and more elementary construction of the complex structure
on  $\teich(f)$ 
(we will notably not use the results of \cite{mcmullen1988quasiconformal}).

A key tool for this construction is the following analytical result
on quasiconformal vector fields (see definition
\ref{def:champqc}), which is interesting in its own right.

\begin{thmA}
Let $\Omega$ be a hyperbolic open subset of $\rs$ and $\xi$ be a quasiconformal vector 
field on $\Omega$. 
The following properties are equivalent :
\begin{itemize}
\item[$i)$]We have $\rho_\Omega(\xi) \in L^\infty(\Omega)$.
\item[$ii)$] We have $\|\rho_\Omega(\xi)\|_{L^\infty(\Omega)} \leq
 4 \| \dbar \xi \|_{L^\infty(\Omega)}$.
\item[$iii)$] There exists a quasiconformal extension $\hat{\xi}$ of $\xi$ on 
all of $\rs$ with $\hat{\xi}=0$ on $\partial \Omega$.
\item[$iv)$] The extension $\hat{\xi}$ defined by $\hat{\xi}(z) = \xi(z)$ if $z \in \Omega$
 and $0$ else is quasiconformal on 
$\rs$, and $\dbar \hat{\xi}(z)=0$ for almost every $z \notin \Omega$. 
\end{itemize}
\end{thmA}

In particular, we get a new characterization of infinitesimally trivial Teichmüller differentials
on hyperbolic Riemann surfaces (see definition \ref{def:inftrivial}) :

\begin{coro}
A Beltrami differential $\mu$ on a hyperbolic Riemann surface $\s$ is infinitesimally trival
if and only if it is of the form $\mu = \dbar \xi$, with
$\|\rho_\s(\xi)\|_{L^\infty(\s)} \leq  4\| \dbar \xi\|_{L^\infty(\s)}$,
where $\rho_\s$ is the hyperbolic metric on $\s$.
\end{coro}

We will also get a simplified proof of Bers' theorem on the density of rational quadratic differentials,
which notably doesn't use Ahlfors' Mollifier : 

\begin{coro}[Bers' density theorem]
Let $K$ be a compact of $\rs$ containing at least $3$ points, and let $Z$ be  
a countable dense subset of $K$.
The space of meromorphic quadratic differentials with simple poles in $Z$ is dense 
(for the $L^1$ norm) in the space of integrable quadratic differentials which are 
holomorphic outside of $K$.
\end{coro}

The proof of the Main theorem will also yield the following description of the 
tangent and cotangent spaces to the Teichmüller space of $f$ (here, $\Lambda_f$ 
is the closure of the grand orbit of the critical points, $Q(\Lambda_f)$ is the 
space of integrable quadratic differentials holomorphic outside $\Lambda_f$,
and $\nabla_f = \id - f_*$) :

\begin{coro}
We have the following identification :
$$T_0 \teich(f) = \belf/\{\dbar \xi, \xi = f^*\xi \}$$
$$T_0^* \teich(f)= Q(\Lambda_f)/\overline{\nabla_f Q(\Lambda_f)}.$$
\end{coro}

In section \ref{sec:qcvectorfields}, we will be concerned only with non-dynamical, analytic
results on quasiconformal vector fields. The main result of this section is theorem A.
In section \ref{sec:dynteich}, we will apply theorem A to obtain the key fact that $D\Psi^Z$ has 
constant rank. 
Lastly, we will prove the Main Theorem in section \ref{sec:mainproof}.

\section{Quasiconformal vector fields}\label{sec:qcvectorfields}

\subsection{Generalities}

In this section, we introduce notations and recall important results on several mathematical 
objects involved in quasiconformal Teichmüller theory : Beltrami forms and differentials,
quadratic differentials, and quasiconformal vector fields.

In all of the article, $\s$ will denote a Riemann surface.

\begin{defi}
A quadratic differential on $\s$ is a section of the vector bundle $T^* \s \otimes T^*\s$
(symmetric tensor product).
\end{defi}

\begin{defi}
If $\mu$ is a section of $\overline{\mathrm{Hom}}(T\s,T\s)$, i.e. a section of the vector bundle
of anti-$\C$-linear endomorphisms of tangent planes, and $z \in \s$, then  let  $|\mu|(z)$ denote the norm of the endomorphism
$\mu(z)$ of $T_z \s$ : $|\mu|$ is a well defined function on $\s$. If $\mu$ is such a section
verifying $|\mu| \in L^\infty(\s)$, $\mu$ is called a Beltrami differential.
 If additionally  $\| \mu\|_{L^\infty(\s)} <1$, we say that $\mu$ is a Beltrami form.
\end{defi}

\begin{defi}\label{def:champqc}
Let $\xi$ be a vector field on $\s$. We say that $\xi$ is quasiconformal if $\dbar \xi$
 (in the sense of distribution theory)
is a Beltrami differential.
\end{defi}

More generally, it will be useful to define :

\begin{defi}
For $(p,q) \in \Z \times \{0,1\}$, we define $S_p^q$ as the space of sections of
$T^*\s^{\otimes p} \otimes \overline{T^*\s}^{\otimes q}$ if $p \geq 0$, and the space of
sections of $T\s^{\otimes p} \otimes \overline{T^*\s}^{\otimes q}$ if $p<0$
(all tensor products are symmetric).
\end{defi}

Since $E \otimes  \overline{E^*}$ is canonically isomorphic to
$\overline{\mathrm{Hom}}(E,E)$ for all complex vector space $E$, Beltrami differentials are exactly 
the $\mu \in S_{-1}^1$ such that $|\mu| \in L^\infty(\s)$. Similarly, quadratic differentials
are the elements of $S_2^0$, vector fields are the elements of $S_{-1}^0$, 
generalized Beltrami differentials (introduced by Bers in \cite{bers1967inequalities}) are the elements of
$S_{-k}^1$, $k \in \N$, and the differentials of order $p$ are the elements of
$S_p^0$, $p \in \N$.

In local coordinates, elements of $S_p^q$ are written $u=u(z) dz^p d\overline{z}^q$, 
and under a change of coordinates $\phi(w)=z$, we have
$u = u \circ \phi(w) \phi'(w)^p \overline{\phi'(w)}^q dw^p d\overline{w}^q$.

\begin{defi}
Let $u_i = v_i \otimes w_i \in S_{p_i}^{q_i}$,
with $v_i \in (T^* \s)^{\otimes p_i}$ and $w_i \in (\overline{T^* \s})^{\otimes q_i}$
, $1 \leq i \leq 2$. Assume that 
$p_1 \leq p_2$, $0 \leq p_1 + p_2 \leq 1$, and $0 \leq q_1 + q_2 \leq 1$. We then define :
$$u_1 \cdot u_2 =  u_1(u_2, \cdot) \wedge (v_1 \otimes v_2)$$
\end{defi}

 Thus $u_1 \cdot u_2$ is an alternate $(p_1+p_2,q_1+q_2)$ differential form
on $\s$. By convention, we will set $u_2 \cdot u_1 = u_1 \cdot u_2$.

Note that in local coordinates, 
if $u_i = u_i(z) dz^{p_i} d\overline{z}^{q_i}$ with $1 \leq i \leq 2$ and
$p_1 + p_2 \in \{0,1\}$ and $q_1 + q_2 \in \{0,1\}$, then
$$u_1 \cdot u_2 = u_1(z) u_2(z) dz^{p_1 + p_2} d\overline{z}^{q_1 + q_2}.$$

The next definition is a particular case of the usual definition of $\dbar u$, where 
$u$ is a section of a holomorphic vector bundle.

\begin{defi}
Let $u \in S_p^0$ be of class $C^1$. We can write locally $u = \phi  v$,
where $v$ is a holomorphic local section of $(T^*\s)^{\otimes p	}$ and $\phi : \s \rightarrow \C$ is  $C^1$.
 We define $\dbar u \in S_p^1$ by :
$$\dbar u = \dbar \phi \otimes v$$
This definition is independant of the choice of $\phi$ and $v$.
\end{defi}

Note that in local coordinates, if $u = u(z) dz^p$, then
$\dbar u(z) = \frac{\partial u}{\partial \overline{z}}(z) dz^p d\overline{z}$.
Also, it is not hard to see that $\dbar (u \cdot v) = u \cdot \dbar v + \dbar u \cdot v$.

There are three particular cases of the above definitions which are especially 
important and deserve to be explicitly worked out : \newline

$a)$ The case of $q \cdot \mu$, where $q$ is a quadratic differential and $\mu$
is a Beltrami differential :

Then $q \cdot \mu$ is a (1,1) alternate form given by :
\begin{align*}
(q \cdot \mu)_z :\;& T_z \s \times T_z \s \rightarrow \C \\
&(u,v) \mapsto \frac{1}{2}(q_z(u,\mu_z(v)) - q_z(v, \mu_z(u)))
\end{align*}   
In local coordinates, $q \cdot \mu = q(z) \mu(z) dz \wedge d\overline{z}$. \newline

$b)$ The case of $q \cdot \xi$, where $q$ is a quadratic differential and
$\xi$ is a vector field :

Then $q \cdot \xi$ is a  $(1,0)$-differential form, given by :
\begin{align*}
(q \cdot \xi)_z :\;& T_z \s \rightarrow \C \\
&u \mapsto q_z(u,\xi(z))
\end{align*}   
In local coordinates, $q \cdot \xi = q(z) \xi(z) dz $. \newline

$c)$ The case of $\dbar q \cdot \xi$, where $q$ is a quadratic differential and
$\xi$ is a vector field :

We can write locally $q = \phi q'$,
where $\phi$ is a function and $q'$ is a holomorphic quadratic differential.
Then $\dbar q = \dbar \phi \otimes q' \in S_{2}^1$, and $\dbar q \cdot \xi$ 
 is a $(1,1)$ alternate differential form, given by :
\begin{align*}
(\dbar q \cdot \xi)_z :\;& T_z \s \times T_z  \s \rightarrow \C \\
&(u,v) \mapsto \frac{1}{2} (\dbar \phi_z(u) q'_z(v,\xi(z)) - \dbar \phi_z(v) q'_z(u,\xi(z)))
\end{align*}   
In local coordinates, $\dbar q \cdot \mu = \frac{\partial q}{\partial \overline z}(z) \xi(z)
 dz \wedge d\overline{z}$. \newline

\begin{prop}[Stokes' theorem for quasiconformal vector fields]
Let $U$ be an open subset of $\rs$ with piecewise $C^1$ boundary, let $q$ be a  $C^1$
 quadratic differential continuous on $\overline{U}$ 
and $\xi$ a quasiconformal vector field on $\rs$. Then 
$$\int_U q \cdot \dbar \xi +\int_U \xi \cdot \dbar q = \int_{\partial U} q \cdot \xi$$
\end{prop}

\begin{proof}
In the case where $\xi$ is a $C^1$ vector field, this is
exactly the classical Stokes' theorem. We deduce the general case where $\dbar \xi$ 
only exists in the sense of distribution 
with a density argument : let $\xi$ be a quasiconformal vector field and $\xi_n$
a sequence of vector fields which are $C^1$ in the neighborhood of $\overline{U}$ and converging 
uniformly to $\xi$ on $\overline{U}$ (such a sequence exists because $\xi$ is continuous).
 Then $\xi_n$ converges to $\xi$ as a distribution on $U$, so
$\dbar \xi_n$ converges to $\dbar \xi$ in the sense of distributions (by continuity of the
$\dbar$ operator for the topology of distributions). Since we know that $\dbar \xi$ 
is in fact a $L^\infty$ Beltrami differential, we deduce from this that for all 
test quadratic differential $\phi$ (i.e. smooth and with compact support in $U$), we have :
$$\limn \int_U \phi \cdot \dbar \xi_n = \int_U \phi \cdot \dbar \xi$$
Since test quadratic differentials are dense for the $L^1$ norm
, this still holds for all quadratic differential $\phi$ integrable on $U$, 
and in particular for $q$. 

Therefore $\limn \int_U q \cdot \dbar \xi_n = \int_U q \cdot \dbar \xi$ and since $\xi_n$ converges
uniformly on $\overline{U}$, we also have 
$$\limn -\int_U \xi_n \cdot \dbar q + \int_{\partial U} q \cdot \xi_n  = 
-\int_U \xi \cdot \dbar q + \int_{\partial U} q \cdot \xi.$$
\end{proof}

\begin{defi}
For a rational map $f : \rs \rightarrow \rs$, we will note $\Delta_f = \id - f^*$ 
and $\nabla_f = \id - f_*$, where $f^*$ et $f_*$ are respectively the pullback by $f$
on vector fields (and Beltrami differentials), and $f_*$ is the pushforward by $f$
on quadratic differentials, following the notations of \cite{epstein2009transversality}. 
\end{defi}

\begin{defi}
Let $z_0 \in \s$ and  $\xi(z_0) \in T_{z_0} \s$. If $q$ is
a meromorphic quadratic differential with a simple pole at $z_0$,
we define the residue of $q \cdot \xi$ at $z_0$ as the residue of
$q \cdot \tilde{\xi}$ at $z_0$, where $\tilde{\xi}$ is a vector field holomorphic
in the neighborhood of $z_0$ with $\xi(z_0)=\tilde{\xi}(z_0)$. This definition
does not depend upon the choice of $\tilde{\xi}$.
\end{defi}

\begin{prop}\label{prop:integrale.residu}
Let $q$ be a meromorphic quadratic differential on an open domain
$\Omega$ with smooth boundary, relatively compact in a Riemann surface $\s$, 
with simple poles that are included in a finite set $P$.
 Let $\xi$ be a quasiconformal vector field on $\Omega$
 extending continuously to $\overline{\Omega}$. Then :
$$\int_{\Omega} q \cdot \dbar \xi = 2 i \pi \sum_{z \in P} \res(q \cdot \xi, z) - \int_{\partial \Omega} q \cdot \xi$$
\end{prop}

\begin{proof}
Let $\Omega_\epsilon = \Omega - \cup_{z \in P} D(z,\epsilon)$
where $D(z,\epsilon)$ is the closed disk  of center $z$ and radius $\epsilon$
(for an arbitrary metric). Then, by Stokes' theorem, 
$$\int_{\Omega_\epsilon} q \cdot \dbar \xi = -\int_{\partial \Omega_\epsilon} q \cdot \xi = -\int_{\partial \Omega} q \cdot \xi
+ \sum_{z \in P} \int_{\partial D(z,\epsilon)} q \cdot \xi$$ 

Let $\xi_\epsilon$ be a quasiconformal vector field coinciding with $\xi$ on $P$ and on $\Omega_\epsilon$
and holomorphic in the neighborhood of $P$. Then : 
$$\int_{\Omega_\epsilon} q \cdot \dbar \xi_\epsilon = 2 i \pi \sum_{z \in P} \res(q \cdot \xi, z) - \int_{\partial \Omega} q \cdot \xi$$

Since $\int_{\Omega} q \cdot \dbar \xi - \int_{\Omega_\epsilon} q \cdot \dbar \xi_\epsilon = O(\epsilon)$,
the result follows by letting $\epsilon$ tend to zero.
\end{proof}

Note that in the particular case $\Omega = \Delta$ and
$q = \frac{dz^2}{z}$, we get the usual Cauchy-Pompéiu formula.

\subsection{Splitting and hyperbolic metric}

\begin{defi}
Let $\s$ be a hyperbolic Riemann surface and $\xi$ a vector field on
$\s$. We say that $\xi$ is hyperbolically bounded on $\s$ 
if and only if $\rho_\s(\xi) \in L^\infty(\s)$, where
$\rho_\s$ is the hyperbolic metric on $\s$.
\end{defi}

\begin{thm}\label{th:hypbonimpliquerecol}
Let $\xi$ be a vector field hyperbolically bounded
on an open hyperbolic subset
 $\Omega$ of $\rs$, quasiconformal on $\Omega$ and identically vanishing outside $\Omega$.
Then $\xi$ is globally quasiconformal,
and $\dbar \xi=0$ for almost every $z \notin \Omega$. 
Moreover, $\|\rho(\xi)\|_{L^\infty(\Omega)} \leq 4\|\dbar \xi \|_{L^\infty(\Omega)}$.
\end{thm}

\begin{proof}

The key point is the following lemma :

\begin{lem}\label{lem:stokesahlfors}
Let $q$ be an integrable quadratic differential of class $C^\infty$ on $\Omega$, and $\xi$ 
a hyperbolically bounded quasiconformal vector field on $\Omega$. 
Assume that $\xi \cdot \dbar q$ is integrable on $\Omega$.
Then :
$$\int_{\Omega} \dbar \xi \cdot q  =  - \int_{\Omega} \xi \cdot \dbar q $$
\end{lem}

\begin{proof}

Let $z_0 \in \Omega$ be an arbitrary base point, and let $\delta(z)=d_\Omega(z,z_0)$,
where $d_\Omega$ is the hyperbolic distance on $\Omega$.
Let $\phi : \R^+ \rightarrow \R^+$ be a smooth function such that $\phi(x)=1$ for $x \in [0,1]$
and $\phi(x)=0$ for $x \geq 2$. 
For all $n \in \N$, let us define $\phi_n : \Omega \rightarrow \R^+$ by 
$$\phi_n(z)=\phi\left(  \frac{\delta(z)}{n}\right).$$
Set $\mu = \dbar \xi$.

Let $q$ be a quadratic differential as in the statement of the lemma. 
Since $\phi_n q$ is compactly supported in $\Omega$, we have :
\begin{equation*}
\int_{\Omega} \mu \cdot (\phi_n q) = - \int_{\Omega} \xi \cdot \dbar (\phi_n q)= \int_{\Omega} \xi \cdot (\phi_n \dbar q) + \xi \cdot (q \cdot \dbar \phi_n) 
\end{equation*}

Moreover, $\int_{\Omega} \xi \cdot (q \cdot \dbar \phi_n) = \int_{\Omega} q \cdot (\xi \cdot \dbar \phi_n)$.
Let us now evaluate the $L^\infty$ norm of the Beltrami differential $\dbar \phi_n \cdot \xi$. 
Since $\delta$ is a locally lipschitz function on $\Omega$, it has locally bounded distributional
derivatives. We have :
\begin{align*}
\dbar \phi_n  &= \frac{1}{n} \phi'(\delta/n) \dbar \delta.
\end{align*}

Let $z \in \Omega$ and $u \in T_z \rs$. 
We have $\dbar \phi_n \cdot \xi(z) : u \mapsto  \dbar \phi_n(u) \xi(z)$, and the norm of this endomorphism
for any hermitian metric is $| \dbar \phi_n \cdot \xi|(z)$. We can therefore work with the hyperbolic metric $\rho_\Omega$. 
Since $\delta$ is 1-lipschitz for the hyperbolic metric in $\Omega$, the derivative $\dbar \delta$ has hyperbolic
norm less than one almost everywhere.
We have : 
\begin{align*}
 \dbar \phi_n \cdot \xi(z; u) &= \frac{1}{n} \phi'\left(\frac{\delta(z)}{n}\right) \dbar \delta(z; u)  \times \xi(z)   \\
\rho_\Omega (\dbar \phi_n \cdot \xi(z; u) )   &\leq \frac{\sup_{\R^+} |\phi'|}{n} \|\rho_\Omega(\xi)\|_{L^\infty(\Omega)} \rho_\Omega(u) \\
|\dbar \phi_n \cdot \xi(z)| &\leq \frac{\sup_{\R^+} |\phi'|}{n} \|\rho_\Omega(\xi)\|_{L^\infty(\Omega)}.
\end{align*}
Therefore : $\| \dbar \phi_n \cdot \xi \|_{L^\infty} = O(1/n)$ and
$\left| \int_{\Omega} q \cdot (\xi \cdot \dbar \phi_n) \right| \leq \|q \|_{L^1} \| \dbar \phi_n \cdot \xi \|_{L^\infty} = O(1/n)$.

We then have :
\begin{equation*}
\int_{\Omega} \dbar \xi \cdot (\phi_n q) = \int_{\rs} (\dbar \xi \cdot q) \phi_n  =  - \int_{\Omega}  \phi_n (\xi \cdot \dbar q) + O(1/n)
\end{equation*}
so

\begin{equation*}
\int_{\Omega} (\dbar \xi \cdot q) \phi_n  =  - \int_{\Omega}  \phi_n (\xi \cdot \dbar q) + O(1/n)
\end{equation*}

Since we assumed that both $\xi \cdot \dbar q$ and $|q|$ are integrable, we 
can apply the dominated convergence theorem to get :
\begin{equation*}\label{eq:integrpartie}
\int_{\Omega} \dbar \xi \cdot q  =  - \int_{\Omega} \xi \cdot \dbar q .
\end{equation*}
\end{proof}

Let now 
$q$ be a $C^\infty$ quadratic differential on $\rs$ : 
its restriction to $\Omega$ verifies the conditions of the lemma, therefore we have :

$$\int_{\rs} \mu \cdot q  =  - \int_{\rs} \xi \cdot \dbar q,$$
where $\mu = \dbar \xi$ on $\Omega$ and $0$ elsewhere.
This means precisely that $\dbar \xi = \mu$ in the sense of distributions
on $\rs$, 
which proves the first assertion of the theorem.

Let us now prove the second assertion. Denote by $\tilde{\xi} = p^* \xi_{|\Omega}$
where 
$p : \Delta \rightarrow \Omega$
is a universal cover of $\Omega$ mapping $0$ to an arbitrary point $z_0 \in \Omega$. 
Proposition \ref{prop:integrale.residu} applied to $\tilde{\xi}$ and 
$q=\frac{dz^2}{z}$ on $\Delta$ yields :
$$\res\left(\frac{dz^2}{z} \cdot \tilde{\xi}(0), 0\right)= dz(\tilde{\xi}(0)) = 
\frac{1}{2i\pi}\int_{\Delta_r} \dbar \tilde{\xi}(z) \cdot \frac{dz^2}{z} +
 \frac{1}{2i\pi} \int_{S_r} \tilde{\xi}(z) \cdot \frac{dz^2}{z}$$
where $\Delta_r$ and $S_r$ are respectively the disk of radius $r$ and the circle of radius $r$.
Since we assumed that  
$\|\rho_\Omega(\xi)\|_{L^\infty(\Omega)}= \|\rho_\Delta(\tilde{\xi})\|_{L^\infty(\Delta)}$ is finite, the second term converges to $0$ when
 $r$ tends to $1$. 
Therefore, by letting $r$ converging to $1$ :
$$dz(\tilde{\xi}(0)) = \frac{1}{2i\pi}\int_{\Delta} \dbar \tilde{\xi}(z) \cdot \frac{dz^2}{z} $$
and
$$\rho_\Delta(\tilde{\xi})(0) = 2|\tilde{\xi}(0)| \leq
 2\frac{1}{2\pi}\|\frac{dz^2}{z}\|_{L^1(\Delta)} \|\dbar p^* \xi \|_{L^\infty}.$$

Since $\|\frac{dz^2}{z}\|_{L^1(\Delta)} = 4\pi$ et $\dbar p^* \xi = p^* \dbar \xi$, we deduce 
$$\rho_\Delta(\tilde{\xi})(0) = \rho_\Omega(\xi)(z_0) \leq 4\| \dbar \xi\|_{L^\infty{(\Omega)}}.$$
Since $z_0$ is arbitrary, this concludes the proof of the second assertion.
\end{proof}

This last theorem states that if we have a hyperbolically bounded quasiconformal vector 
field on an open set $\Omega$, we can glue it together with the zero vector field outside
$\Omega$ and still get a globally quasiconformal vector field. The next proposition gives a little
more than the converse. We will need the following lemma :

\begin{lem}\label{lem:limmétrhyp}
Let $\Omega$ be a hyperbolic open subset of $\rs$, and $X$ a countable dense subset of 
$\partial \Omega$. Let $(X_n)$ be an increasing sequence of finite subsets of $X$ 
with $\cup_n X_n = X$ and
$\mathrm{card} X_n \geq 3$ for all $n \in \N$, and note $\Omega_n = \rs - X_n$. Then the 
hyperbolic metric $\rho_{\Omega_n}$ of $\Omega_n$ converges pointwise on $\Omega$ 
to the hyperbolic metric $\rho_\Omega$ of $\Omega$.
\end{lem}

The proof is not difficult and makes use of Montel's theorem and the Schwarz lemma 
applied to the inclusions
$\Omega \hookrightarrow \Omega_n$.

\begin{prop}\label{prop:hypborne}
Let $\xi$ be a quasiconformal vector field on $\rs$ vanishing on the boundary
of a hyperbolic open subset $\Omega$ of $\rs$.
Then :
$$\|\rho_\Omega(\xi)\|_{L^\infty(\Omega)} \leq 4\| \dbar \xi \|_{L^\infty(\Omega)}$$
\end{prop}

\begin{proof}
Denote by $K$ the boundary of $\Omega$.
Let $(X_n)_{n \in \N}$ be an increasing sequence of finite subsets  of $\partial \Omega$
whose union is dense in
$\partial \Omega$, with $\mathrm{card} X_n \geq 3$.
 Then by lemma ~\ref{lem:limmétrhyp}, the hyperbolic metric $\rho_{\Omega_n}$ of 
 $\Omega_n = \rs - X_n$ converges 
pointwise to the hyperbolic metric $\rho_\Omega$ of $\Omega$ on $\Omega$. 
By Theorem \ref{th:hypbonimpliquerecol}, it then suffices to show that for all $n \in \N$,  
$\| \rho_{\Omega_n}(\xi)\|_{L^\infty(\Omega)}$ is bounded.

Therefore it is enough to show the weaker property : for all $n \in \N$, 
there exists a constant
$C_n>0$ such that $\sup_\Omega \rho_{\Omega_n}(\xi) \leq C_n$. Since $\rho_{\Omega_n}(\xi)$ 
is a continuous function on $\Omega_n = \rs - X_n$, it is enough to show that  $\rho_{\Omega_n}(\xi)$ 
is bounded in the neighborhood of all $z \in X_n$
(by a constant depending for now on $n \in \N$). Let $z_0 \in X_n$, and $r>0$ such that
the punctured disk $U$ of center $z_0$ and radius $r$ is included in $\Omega_n$. 
Then by the Schwarz lemma,
the hyperbolic metric of $\Omega_n$ is smaller than that of $U$, so we have for all
 $z \in U$ :
$$\rho_{\Omega_n}(\xi)(z) \leq \rho_U(\xi)(z) \leq C'_n |\xi(z)| \left( |z-z_0| \log |z-z_0|^{-1} \right)^{-1}.$$
The second inequality is a classical estimate of the hyperbolic metric of the 
punctured disk in the neighborhood of $z_0$ (see for example
\cite{gardiner2000quasiconformal} or
\cite{hubbard2006teichmuller}). The constant $C_n'$ still depends a priori on $r$ and 
therefore on $n$. 
Furthermore, $\xi$ has a continuity modulus on $-\epsilon \log \epsilon$ by virtue of
quasiconformality
(cf \cite{gardiner2000quasiconformal}, theorem $7$ p. $56$), so there exists a constant
 $C>0$ (depending only on $\xi$ and on the choice of coordinates) such that in the 
 coordinates $z$ : 
$$|\xi(z)| = |\xi(z) - \xi(z_0)| \leq C |z-z_0| \log |z-z_0|^{-1}.$$
We therefore have, for all $z \in D_r(z_0)$ : 
$$\rho_{\Omega_n}(\xi)(z) \leq C_n.$$

The Theorem \ref{th:hypbonimpliquerecol} applied to $\xi$ on $\Omega_n$ then allows us 
to get a uniform bound with respect to $n$ :
$$ \| \rho_{\Omega_n}(\xi) \|_{L^\infty(\Omega_n)} \leq 4\| \dbar \xi \|_{L^\infty(\Omega_n)} 
\leq 4\| \dbar \xi \|_{L^\infty(\rs)}.$$

By passing to the limit, we get :
$$ \| \rho_{\Omega}(\xi) \|_{L^\infty(\Omega)} \leq 4 \| \dbar \xi \|_{L^\infty(\rs)},$$
and a second application of the same theorem finally yields : 
$$\| \rho_{\Omega}(\xi) \|_{L^\infty(\Omega)} \leq 4\| \dbar \xi \|_{L^\infty(\Omega)}.$$

\end{proof}

By combining the results of Theorem \ref{th:hypbonimpliquerecol} and proposition 
\ref{prop:hypborne}, we get :

\begin{thmA}
Let $\Omega$ be a hyperbolic open subset of $\rs$ and $\xi$ be a quasiconformal vector 
field on $\Omega$. 
The following properties are equivalent :
\begin{itemize}
\item[$i)$]We have $\rho_\Omega(\xi) \in L^\infty(\Omega)$.
\item[$ii)$] We have $\|\rho_\Omega(\xi)\|_{L^\infty(\Omega)} \leq
 4 \| \dbar \xi \|_{L^\infty(\Omega)}$.
\item[$iii)$] There exists a quasiconformal extension $\hat{\xi}$ of $\xi$ on 
all of $\rs$ with $\hat{\xi}=0$ on $\partial \Omega$.
\item[$iv)$] The extension $\hat{\xi}$ defined by $\hat{\xi}(z) = \xi(z)$ if $z \in \Omega$
 and $0$ else is quasiconformal on 
$\rs$, and $\dbar \hat{\xi}(z)=0$ for almost every $z \notin \Omega$. 
\end{itemize}
\end{thmA}

\begin{coro}\label{coro:decoupage}
Let $\Omega$ be a hyperbolic open subset of $\rs$ and $\xi$ be a quasiconformal vector field 
vanishing on $\rs - \Omega$. Let $\Omega = \bigsqcup_i \Omega_i$ a countable partition of 
$\Omega$ into open sets $\Omega_i$. Then
$$\xi = \sum_i \xi_i$$
where $\xi_i$ is a quasiconformal vector field coinciding with $\xi$ on $\Omega_i$ and 
vanishing outside $\Omega_i$.
\end{coro}

\begin{proof}
By item $iv)$ of theorem A, the vector fields $\xi_i$ are quasiconformal.
\end{proof}

Recall the following notion, which is of importance in Teichmüller theory : 

\begin{defi}\label{def:inftrivial}
A Beltrami differential $\mu$ on a Riemann surface $\s$ is infinitesimally trivial
if $\int_{\s} q \cdot \mu = 0$ for all quadratic differential $q$ holomorphic on $\s$.
\end{defi}

The terminology comes from the fact that the tangent space to the base point
$T_0 \teich(\s)$ identifies canonically to the quotient of the space of Beltrami differentials
on $\s$ by the space of infinitesimally trivial Beltrami differentials
(see \cite{gardiner2000quasiconformal} or
 \cite{hubbard2006teichmuller}).

The next result is a theorem due to Bers. Its proof classically involves a 
delicate mollifier introduced by Ahlfors, the so-called Ahlfors Mollifier,
see \cite{gardiner2000quasiconformal}, theorem $9$ p. 63. The mollifier
$\phi_n$ of the proof of theorem \ref{th:hypbonimpliquerecol} 
replaces the Ahlfors Mollifier and yields a simplified proof.

\begin{coro}[Bers density theorem]
Let $K$ be a compact of $\rs$ containing at least $3$ points, and $A$ a
countable dense subset of $K$.
The space of meromorphic quadratic differentials with simple poles in $A$ is dense
 (for the $L^1$ topology) in the space of integrable quadratic differentials on $\rs$
 which are holomorphic outside of $K$.
\end{coro}

\begin{proof}
It is enough to show that any continuous linear form on the space of integrable
quadratic differentials holomorphic outside $\rs$ vanishing against all meromorphic 
quadratic differentials with only simple poles in $A$ must be trivial.
By the Hahn-Banach theorem, any such linear form may be represented by a $L^\infty$ 
Beltrami differential on $\rs$. Let $\mu$ be such a Beltrami differential and 
$\xi$ a quasiconformal vector field such that $\mu = \dbar \xi$, and assume that  
$$\int_{\rs} q \cdot \dbar \xi=0$$
for all meromorphic integrable quadratic differential $q$ with simple poles in $A$. 
Let $Z \subset A$ a set of cardinal $3$ : by adding to $\xi$ a holomorphic vector field,
we lose no generality by assuming that $\xi$ vanishes on $Z$. Then by proposition 
~\ref{prop:integrale.residu} applied to $\Omega = \rs$ and $q$ a quadratic differential
with simple poles precisely in $Z$ and at $z \in A \backslash Z$, one sees that  $\xi$
must vanish at $z$. By continuity, $\xi$ vanishes on all of $K$.
 So by theorem A, $\xi$ is hyperbolically 
bounded on $\Omega$.
Let $q$ be an integrable quadratic differential that is holomorphic on $\Omega$. In particular,
$q$ is $C^\infty$ and integrable on $\Omega$, and $\dbar q$ vanishes on $\Omega$.
Lemma \ref{lem:stokesahlfors} yields :
$$\int_{\Omega} q \cdot \dbar \xi = - \int_{\Omega} \dbar q \cdot \xi =0.$$
Moreover, by theorem A, we have $\dbar \xi = 0$ almost everywhere on $K$, so :
$$\int_K q \cdot \dbar \xi = 0,$$
which ends the proof.
\end{proof}

\begin{coro}
Let $\Omega$ be a hyperbolic open subset of $\rs$, and $\mu$ be a Beltrami differential on 
$\Omega$. Then $\mu$ is infinitesimally trivial if and only if there exists a 
hyperbolically bounded quasiconformal vector field $\xi$ on $\Omega$ such that 
$\mu = \dbar \xi$. 
\end{coro}

\begin{proof}
We just proved that a Beltrami differential $\mu$ is infinitesimally trivial 
on $\Omega$
if and only if there exists a quasiconformal vector field $\xi$ on $\rs$ such that 
$\mu = \dbar \xi$ on $\Omega$ and 
$\xi = 0$ on $\rs - \Omega$.  By theorem A, this property is equivalent to
being hyperbolically bounded in $\Omega$.
\end{proof}

\section{Dynamical Teichmüller space}\label{sec:dynteich}

\subsection{The differential of $\Psi$}

If $\lambda \mapsto f_\lambda$ is a holomorphic curve in $\ratd$ passing through $f_0=f$, then 
$\dot f = \frac{d f_\lambda}{d\lambda}_{|\lambda=0}$ is a section of the bundle $f^* T\rs$, and
 $Df^{-1} \circ \dot f$ is a meromorphic vector field on $\rs$, whose poles are included in
  $\mathrm{Crit}(f)$ and of multiplicity at most
that of the critical points of $f$. Denoting by $T(f)$ the complex vector space of such 
vector fields, 
we obtain a canonical identification between  $T_f \ratd$ and $T(f)$. In the rest 
of this artical, we will implicitly identify 
$T_f \ratd$ with $T(f)$.

Denote as well by $\mathrm{aut}(\rs)$ the space of holomorphic vector fields on $\rs$ and by 
$\of$ the orbit of $f$ by conjuacy via Möbius transformation.
By \cite{buff2009bifurcation}, proposition 1, $\of$ is a complex submanifold of $\ratd$ 
of dimension $3$,
and $T_f \of = \Delta_f \aut(\rs) \subset T(f)$.

\begin{prop}\label{prop:diffphi}
Let $\xi$ be a quasiconformal vector field on $\rs$ such that $\dbar \xi \in \belf$.
Then $\Delta_f \xi \in T(f)$.  Moreover, if we assume that $\xi$ vanishes on a set $Z$ of
cardinal $3$,
then :
$$D\Psi^Z(0)\cdot \dbar \xi = -\Delta_f \xi.$$
\end{prop}

\begin{proof}
An easy calculation shows that for almost every $z \notin \critf$, $\dbar f^*\xi = f^*\dbar \xi$.
Therefore by Weyl's lemma, $\Delta_f \xi = \xi - f^*\xi$ is holomorphic  
on $\rs - \critf$. Since $\xi$ is continuous, we have $\Delta_f \xi = O(1/f')$ in the 
neighborhood of  $\critf$, so $\Delta_f \xi$ has at every critical point  $c$ of $f$ a 
pole of at most the multiplicity of $c$ cas a critical point of $f$ ; so $\Delta_f \xi \in T(f)$.

Moreover, if $\mu_\lambda \in \belf$ is a holomorphic curve passing through $0$,
with $\mu_\lambda = \lambda \dbar \xi + o(\lambda)$, then we have :
$$\phi_{\mu_\lambda}^Z = \mathrm{Id} + \lambda \xi + o(\lambda)$$
where $\phi_{\mu_\lambda}^Z$  is the unique quasiconformal homeomorphism
associated to $\mu_\lambda$ fixing $Z$ (see \cite{gardiner2000quasiconformal} or
\cite{hubbard2006teichmuller}).
If we differentiate with respect to $\lambda$ the equality 
$$\phi_{\mu_\lambda}^Z \circ f = f_\lambda \circ \phi_{\mu_\lambda}^Z,$$
we get :
$$\xi \circ f = \dot f + Df(\xi),$$
où $\dot f = \frac{d f_\lambda}{d\lambda}_{|\lambda = 0}$. This can be rewritten as :
$$\eta := Df^{-1} (\dot f) = - \Delta_f \xi.$$
\end{proof}

With an abuse of notations, we will note $D \Psi(0) : \belf \rightarrow T(f)/T_f \of$
the quotient of the linear application $D \Psi^Z(0) : \belf \rightarrow T(f)$.
This application does not depend on the choice of $Z$.

\begin{defi}
Let $f$ be a rational map. We will note $\Lambda_f$ the closure of the grand critical
orbit of $f$, and $\Omega_f = \rs - \Lambda_f$.
\end{defi}

\begin{prop}\label{prop:noyaudiffphi}
Let $\xi$ be a quasiconformal vector field on $\rs$ such that $\dbar \xi \in \belf$.
The following properties are equivalent :
\begin{itemize}
\item[$i)$] $\dbar \xi \in \ker D{\Psi}(0)$
\item[$ii)$] $\Delta_f \xi \in \Delta_f \aut(\rs)$
\item[$iii)$] There exists $h \in \mathrm{aut}(\rs)$ such that $\xi - h$ vanishes on
$\mathrm{Crit}(f)$ with at least the multiplicity of each critical point of $f$
\item[$iv)$] There exists $h \in \aut(\rs)$ such that $\xi - h$ vanishes on $\Lambda_f$

\end{itemize}
\end{prop}

\begin{proof}
The first two items are equivalent by \cite{buff2009bifurcation}, 
proposition 1.

$ii) \Rightarrow iii)$ : if $\Delta_f \xi = \Delta_f h$, $h \in \aut(\rs)$, 
then $\xi - h$ is a continuous $f$-invariant vector field. Hence
$\xi - h$ must vanish on $\critf$ with at least the multiplicity of the 
critical points of $f$.

$iii) \Rightarrow ii)$ If $\xi - h$ vanishes on $\critf$ with
at least the multiplicity of the critical points of $f$, then
$f^*(\xi - h)$ is well-defined and continuous at $\critf$. 
By the above proposition, $\Delta_f(\xi - h) \in T(f)$,
so $\Delta_f(\xi - h)=0$.

$iv) \Rightarrow ii)$ : If $\xi -h$ vanishes on $\Lambda_f$, 
then since $\Lambda_f$ is invariant $\Delta_f(\xi-h)$ vanishes
as well on $\Lambda_f$. Therefore $\Delta_f(\xi-h)$ is a meromorphic
vector field (by the above proposition) vanishing on 
$\Lambda_f$ which is not discrete, so $\Delta_f(\xi-h)=0$ by
the isolated zeros principle.

$ii) \Rightarrow iv)$ : If $\Delta_f(\xi-h)=0$, then we saw that 
$\xi-h$ must vanish on $\critf$ (item $iii)$).
Therefore $(f^k)^*(\xi-h)(c)=(\xi-h)(c)=0$ for all $k \geq 0$.
Moreover, if $f^p(z)=c \in \critf$, then $(f^p)^* (\xi-h)(z)=0=(\xi-h)(z)$.
So $(\xi-h)$ vanishes on the grand critical orbit of $f$, 
hence on $\Lambda_f$ by continuity.
\end{proof}

Note that if we normalize $\xi$ by imposing the condition that it vanishes on 
on a set $Z$ invariant by $f$ of cardinal $3$, then proposition
\ref{prop:noyaudiffphi} remains true by replacing $h$ by $0$ in items
 $ii)$, $iii)$ and $iv)$, and $D\Psi(0)$ by $D\Psi^Z(0)$
in item $i)$.

We will also need to know the differential $\Psi^Z$ in an arbitrary point
of $\Belf$. Recall the following fact of Teichmüller theory (see \cite{hubbard2006teichmuller}) : 

\begin{defi}
Let $\psi$ be a quasiconformal homeomorphism of $\rs$. For all Beltrami form $\mu$, 
note $\psi^* \mu$ the Beltrami form corresponding to $\phi_{\mu} \circ \psi$,
where $\phi_{\mu}$ is a quasiconformal homeomorphism associated to $\mu$.
\end{defi}

We will also note $\psi_* = (\psi^{-1})^*$.

\begin{prop}
For all quasiconformal homeomorphism $\psi$, the map $\psi^*$ is
biholomorphic.
\end{prop}

We shall need to consider here maps 
$\Psi_f^Z : \Belf \rightarrow \ratd$
and $\Psi_g^Z : \mathrm{Bel}(g) \rightarrow \ratd$ associated to different rational maps
$f$ and $g$. In the rest of the article, there will be no ambiguity and
and we will just use the notation $\Psi^Z$.

\begin{prop}\label{lem:ranguniforme}
Let $\mu \in \Belf$ and $\psi$ the unique corresponding quasiconformal homeomorphism 
fixing $Z$.
Let $g = \psi \circ g \circ \psi^{-1}$. Then 
$$D\Psi_f^Z(\mu) = D\Psi_g^Z(0) \circ D\psi_*(\mu)$$
In particular, $\mathrm{rg} D\Psi_f^Z(\mu) = \mathrm{rg}  D\Psi_g^Z(0)$.
\end{prop}

\begin{proof}
Remark that for all $\phi_0$ and $\phi_\lambda$ associated to elements $\mu_\lambda$ and
$\mu_0$ of $\belf$ :
$$\phi_\lambda \circ f \circ \phi_\lambda^{-1} = 
(\phi_{\lambda} \circ \phi_0^{-1})  \circ \phi_0 \circ f \circ  \phi_0^{-1} \circ (\phi_{\lambda} \circ \phi_0^{-1})^{-1} $$
which may be rewritten as :
$$\Psi_f^Z(\mu_\lambda) = \Psi_g^{Z}(\phi_* \mu_{\lambda})$$
if we assume additionally that $\phi_\lambda$ and $\phi_0$ fix $Z$.
Then we only need to take a curve $\mu_\lambda$ in $\belf$, and 
to differentiate at $\lambda = 0$.
\end{proof}

\subsection{Constant rank theorem in Banach spaces}

Recall the following version of the constant rank theorem in infinite dimension :

\begin{thm}[Constant rank theorem]
Let $\Psi : U \rightarrow F$ be an analytic map, where $U$ is an open subset of a complex
Banach space $E$ and $F$ is a complex finite-dimensional vector space.
Assume that $\mathrm{rg} D\Psi = r$ is constant on $U$. Then for every $x_0\in U$, there 
exists a germ of analytic diffeomorphism $\chi : (F, f(x_0) ) \rightarrow
(F,f(x_0))$ and a germ of analytic diffeomorphism
$\phi : \mathrm{Im}D\Psi(x_0) \oplus \ker D\Psi(x_0) \rightarrow E$ such that for all
$(u,v) \in \mathrm{Im}D\Psi(x_0) \oplus \ker D\Psi(x_0)$ in the neighborhood of $\phi^{-1}(x_0)$,
$$\chi \circ \Psi \circ \phi(u,v)=u.$$
\end{thm}

\begin{coro}
Let $\Psi : E \rightarrow F$ verifying the requirements of the above theorem. Then for all 
$z_0 \in \Psi(E)$, the level set $M = \Psi^{-1}(z_0)$ is a Banach submanifold of $E$, of
codimension $r$ and whose tangent space at $x_0 \in \Psi^{-1}(z_0)$ is  
$T_{x_0} M = \ker D\Psi(x_0)$.
\end{coro}

\begin{proof}
With the notations of the constant rank theorem, we have $\Psi(x)=z_0$ if and only if
$\chi \circ \Psi(u,v)=\chi(z_0) = u$, where $(u,v)=\psi^{-1}(x)$, which is equivalent to
$\psi(\chi(z_0),v))=x$. Since $\psi$ is a (germ of) diffeomorphism, this gives 
a local chart at $x_0$ for $M$, which is therefore a Banach submanifold modeled on
 $\ker D\Psi(x_0)$.
\end{proof}

\subsection{Counting dimensions}

The goal of this section is to show that the differential of
$\Psi^Z : \Belf \rightarrow \ratd$ has constant rank.

\begin{defi}
We say that a critical point is acyclic if it it not preperiodic. We say that two acyclic
critical points lie in the same foliated acyclic critical class if the closure of their
grand orbits are the same.
\end{defi}

The key point to apply the constant rank theorem is the following count of dimension :

\begin{thm}\label{th:calculrang}
Let $f$ be a rational map of degree $d \geq 2$. Then
$$\mathrm{rg} D\Psi(0)=n_f + n_H + n_J  - n_p$$
where $n_H$ is the number of Herman rings of $f$, $n_J$ is the number of ergodic line fields
of $f$, $n_f$ is the number of foliated acyclic critical classes lying in the Fatou set, and
$n_p$ is the number of parabolic cycles.
\end{thm}

\begin{defi}
Let $f : \s \rightarrow \s$ be a holomorphic function.
Denote by $M_f(\s)$ the space of Beltrami forms that are invariant by $f$, and by
$N_f(\s)$ the subspace of $M_f(\s)$ of Beltrami differentials of the form $\dbar \xi$,
where $\xi$ is a hyperbolically bounded quasiconformal vector field on $\s$.
\end{defi}

\begin{thm}\label{th:sommedimension}
Let $f$ be a rational map, and $\Omega$ a hyperbolic open subset of $\rs$ 
 completely invariant under $f$. Let 
$\Omega =  \bigsqcup_i \Omega_i$ be a partition
of $\Omega$ into open subsets $\Omega_i$ completely invariant under $f$. Then 
$$M_f(\Omega)/N_f(\Omega) \simeq \bigoplus_i M(\Omega_i)/N(\Omega_i)$$
\end{thm}

\begin{proof}
Clearly $M_f(\Omega)= \bigoplus_i M(\Omega_i)$. 

Let $\dbar \xi \in N_f(\Omega)$.
By corollary \ref{coro:decoupage}, we have :
$$\xi = \sum_i \xi_i$$
where $\xi_i$ is a quasiconformal vector field coinciding with
$\xi$ on  $\Omega_i$, and such that $\xi_i=0$ outside of $\Omega_i$.
This shows that $N_f(\Omega)=\bigoplus_i N_f(\Omega_i)$.

Hence $M_f(\Omega)/N_f(\Omega) = \bigoplus_i M_f(\Omega_i)/N_f(\Omega_i)$.
\end{proof}

Lastly, we will need the classification of Fatou components, which is a corollary
of Sullivan's no wandering domain theorem. Note that McMullen (see \cite{mcmullen2014notes} ) has given a direct
and purely infinitesimal proof of Sullivan's theorem, which does notably not 
rely on the theory of dynamical Teichmüller spaces. His proof is based on quasiconformal
vector fields and is in the same spirit as the methods used here.

We will also need the following lemmas :

\begin{defi}
Let $M(\s)$ be the set of Beltrami differentials on the Riemann surface $\s$ and 
$N(\s)$ be the subspace of Beltrami differentials on $\s$ that are of the form $\dbar \xi$, 
where $\xi$ is a hyperbolically bounded quasiconformal vector field on $\s$.
\end{defi}

\begin{lem}\label{lem:dicretecase}
Suppose $\Omega$ is the grand orbit of a component of $\Omega_f$ such that
$\Omega/f$ is a hyperbolic Riemann surface. Then the projection $\pi_1 : \Omega \rightarrow \Omega/f$
induces an identification :
$$M_f(\Omega)/N_f(\Omega) \simeq M(\Omega/f)/N(\Omega/f)$$
\end{lem}

\begin{proof}
It is clear that $M_f(\Omega) \simeq M(\Omega/f)$, and that any element 
of $N_f(\Omega)$ passes to the quotient to an element of $N(\Omega/f)$.
 Let $\mu = \dbar \xi \in N(\Omega/f)$.
Since the map $\pi_1 : \Omega \rightarrow \Omega/f$ is a covering between hyperbolic Riemann 
surfaces, it is a local isometry for the hyperbolic metrics, and therefore $\xi_1 = \pi_1^*\xi$
is a hyperbolically bounded quasiconformal vector field, which is invariant by $f$ by construction.
By theorem A, $\hat{\xi_1}$ extended by $0$ outside $\Omega$ is still quasiconformal (and invariant). 
So $\mu = \dbar \hat{\xi_1} \in N_f(\Omega)$.
This proves that $N_f(\Omega) \simeq N(\Omega/f)$.
\end{proof}

\begin{lem}\label{lem:composantepreperiodique}
Let $\Omega \subset \Omega_f$ be an open set completely invariant under $f$
such that all connected component of
$\Omega$ is preperiodic to a component $U$ of period $p \in \N^*$.
Then the restriction to $U$ induces an isomorphism $M_f(\Omega) \rightarrow M_{f^p}(U)$,
mapping $N_f(\Omega)$ onto $N_{f^p}(U)$. In particular,
$$M_f(\Omega)/N_f(\Omega) \simeq M_{f^p}(U)/N_{f^p}(U)$$
\end{lem}

\begin{proof}
Every Beltrami differential $\mu \in M_f(\Omega)$ is invariant under $f$, hence under $f^p$.
Conversely, if $\mu$ is a Beltrami differential on $U$ invariant under $f^p$, then
$\mu$ extends to a Beltrami differential $\tilde{\mu}$ invariant on  $\Omega_i$ 
in the following way : 
if $V$ is a component of $\Omega_i$, then there exists $k \in \N$ (defined up to a multiple of
 $p$ ) such that $f^k_{|V} : V \rightarrow U$.
We then set $\tilde{\mu}_{|V} = (f^k)^*\mu$, and this definition is valid if $V$ belongs 
to the same cycle as $U$ since $\mu=(f^p)^* \mu$.

This identification maps $N_f(\Omega)$ onto $N_{f^p}(U)$ since if 
$\mu = \dbar \xi \in N_f(\Omega)$, then $\hat{\xi}(z)=\xi(z)$ if $z \in U$ and $0$
else is such that $\dbar \hat{\xi} = \mu_{|U}$ by theorem A, and therefore 
$\mu_{|U} \in N_{f^p}(U)$.
\end{proof}

\begin{lem}\label{lem:invarianceparrotation}
Let $\mu$ be a Beltrami differential invariant under a holomorphic function $g$.
In both of the following cases : $g(z)=e^{2i\pi \alpha}z$, $\alpha \notin \mathbb{Q}$, and 
$g(z)=z^d$, $d \geq 2$, $\mu$ is then invariant under all rotations,
and we have in local coordinates  :
$$\mu(r e^{it}) = c(r) e^{2it} \frac{\overline{dz}}{{dz}}$$
\end{lem}

\begin{proof}
The proof is a modification of the usual proof of the ergodicity of rotations
of irrational angles.

Let us start with the case of a rotation of irrational angle $g(z)= e^{2i\pi \alpha}z$.
Let $\mu$ be a Beltrami differential invariant by $g$. We have, in local coordinates :
$$\mu(z)=g^*\mu(z) = e^{-4 i \pi \alpha} \mu(e^{2i\pi \alpha}z)$$
By expanding into Fourier series on the circles $|z|=r$, we obtain
that $\mu$ must be of the form
$$\mu(r e^{it}) = c(r) e^{2it} \frac{\overline{dz}}{{dz}}$$
where $c$ is a $L^\infty$ function. In particular, $\mu$ is invariant by rotations,
and one easily verifies that all rotation-invariant Beltrami differential must 
be of this form.

If we now assume that $g(z)=z^2$, $d \geq 2$, and that $\mu$ is invariant by $g$, then
$\mu$ is invariant by all branches of $g^{-n} \circ g^n$, hence by all
 rotations of angles $\frac{2k\pi}{d^n}$, for all $k \in \N$ :
$$\mu(z) = e^{-4 i k\pi/d^n} \mu(e^{2ik\pi/d^n}z).$$
Similarly, by expanding into Fourier series on the circles centered on $0$, we obtain :
$$\mu(r e^{it}) = c(r) e^{2it} \frac{\overline{dz}}{{dz}}.$$
\end{proof}

\begin{lem}\label{lem:calculdimM/N}
Let $\Omega$ be a rotation invariant planar open set. Let $M(\Omega)$
be the space of rotation-invariant Beltrami differentials on $\Omega$,
and $N(\Omega)$ the subspace of $M(\Omega)$ of elements of the form $\dbar \xi$,
where $\xi$ is a hyperbolically bounded quasiconformal vector field on $\partial \Omega$.
\begin{itemize}
\item[$i)$] If $\Omega$ is the unit disk, $\dim M(\Omega)/N(\Omega)=0$.
\item[$ii)$] If $\Omega$ is a ring of finite modulus, then $\dim  M(\Omega)/N(\Omega)=1$.
\end{itemize}
\end{lem}

\begin{proof}
Consider a vector field $\xi$ of the form
$$\xi(r e^{it}) = h(r) r e^{it} \frac{d}{dz}$$
where $h : \R^+ \rightarrow \R^+$ is a lipschitz function. One can easily verify that 
$$\dbar \xi(r e^{it})= r h'(r) e^{2it} \frac{\overline{dz}}{dz}.$$

Therefore if $\mu$ is a rotation-invariant Beltrami differential,
hence of the form $\mu(r e^{it}) = c(r) e^{2it}\frac{\overline{dz}}{dz}$, 
and if we denote by $h$ the unique primitive
of $r \mapsto c(r)/r$ vanishing at $r=1$ and 
$\xi(r e^{it}) = r h(r) e^{it} \frac{d}{dz}$, we have $\dbar \xi = \mu$ in the sense of 
distributions, and $\xi$ is a quasiconformal vector field on all of $\rs$
vanishing on the unit disk.  

Therefore  $M(\Delta)=N(\Delta)$. If now $\Omega$
denotes a straight ring $\Omega = \{ r_0 <|z| < 1\}$, the map
$$\mu = c(r)e^{2it} \frac{\overline{dz}}{dz} \mapsto h(r_0) = \int_{1}^{r_0} \frac{c(u)}{u}du$$
is a linear form on $M(\Omega)$ whose kernel
is exactly $N(\Omega)$. This linear form is not trivial,
since if we take $\mu = r e^{2it}  \frac{\overline{dz}}{dz}$, then
$h(r_0)=r_0 -1 \neq 0$.
Therefore $\dim M(\Omega)/N(\Omega)=1$.
\end{proof}

We can now prove theorem \ref{th:calculrang}.

\begin{proof}[Proof of theorem \ref{th:calculrang}]
Denote by $\fatou$ the Fatou set of $f$, and $\julia$ its Julia set.
We will also denote by $\mathrm{Fix}_\julia$ the space of invariant line fields.
Since $\ker D\Psi(0) =N_f(\Omega_f)= N_f(\fatou)$ by
proposition \ref{prop:noyaudiffphi}, we have : 
$$\belf/\ker D\Psi(0) = (\mathrm{Fix}_\julia \oplus M_f(\fatou))/N_f(\fatou)$$

If $c$ is a critical point of $f$, then the closure of its grand orbit is equal to the union of
the Julia set $\julia$ and of a countable set of points and smooth circles (if the orbit of $c$ 
is captured by a superattracting cycle, or a cycle of Siegel disks or Herman rings). Therefore
 $\Lambda_f$ coincides with
$\julia$ up to a set of Lebesgues measure zero. Hence $M_f(\fatou)=M_f(\Omega_f)$.
We deduce from this observation that :
$$\mathrm{Fix}(f)/\ker D\Psi(0) = \mathrm{Fix}_\julia \oplus M_f(\Omega_f)/N_f(\Omega_f)$$

Consider the equivalence relationship on the set of connected components of $\Omega_f$
which identifies two components if and only if they have the same grand orbit, and let
$\Omega_i$ be the union of the elements of a class $i$ of this equivalence relationship. The
 $\Omega_i$ form a partition of $\Omega_f$
into completely invariant open subsets. By theorem \ref{th:sommedimension}, we have :
$$\mathrm{rg} D\Psi(0)=\dim \mathrm{Fix}_J + \sum_i \dim  M_f(\Omega_i)/N_f(\Omega_i).$$

Each component $\Omega_i$ is mapped by $f^n$ for $n$ large enough into a periodic Fatou 
component $U$. 
\footnote{However, in the case of a superattracting cycle, the components $\Omega_i$ 
need not be themselves preperiodic : if there is a critical orbit in a superattracting basin,
one gets components $\Omega_i$ which are annuli delimited by equipotentials that accumulate on 
the superattracting cycle.}
Let us now compute $\dim M_f(\Omega_i)/N_f(\Omega_i)$ depending on the nature of the 
periodic Fatou component $U$ it meets.
There are five cases to condider. 
Denote by $n_i$ the number of foliated acyclic critical classes meeting the grand orbit of $U$ \newline

$a)$ The case of an attracting cycle

If $U$ is a component of an attractive basin and $\Omega_i$ meets $U$, then $\Omega_i$ is
the grand orbit of $U$ with the countable set of the critical orbits captured by this cycle (and
 the cycle itslef) removed. So every component of $\Omega_i$ is 
preperiodic to $U - \Lambda_f$.
Thus $f_{|\Omega_i} : \Omega_i \rightarrow \Omega_i$ acts discretely, and 
$X_i = \Omega_i/f$ is a Riemann surface. In a linearizing coordinate for $f^k$
on the immediate basin of attraction (where $k \in \N^*$ is the period of the cycle
and $\rho$ is its multiplier), note 
$A = \{ | \rho | \leq z <1  \}$. It is a fundamental domain for the action of $f$ on 
the cycle of Fatou components $V$ containing $U$,
and $A - \Lambda_f$ is a fundamental domain for the action of $f$ on $\Omega_i$.
Therefore $X_i$ is the torus $X=A/f$ with a finite number $n_i$ of points removed,
where $n_i$ is the number of points of the post-critical set meeting $A$, i.e.
the number of foliated acyclic critical classes meeting $V$.

By lemma \ref{lem:dicretecase}, $\dim M_f(\Omega_i)/N_f(\Omega_i) = \dim M(X_i)/N(X_i)$.
Since $X_i$ is a finitely punctured torus, any hyperbolically bounded quasiconformal vector 
field on $X_i$ extends to a quasiconformal vector field on the torus vanishing on the 
marked points. Then the quotient $ M(X_i)/N(X_i)$ is exactly the tangent space to the 
Teichmüller space of $X_i$, which has dimension equal to the number $n_i$ of marked points
(see for example \cite{hubbard2006teichmuller}). \newline

$b)$ The case of a parabolic cycle

If $U$ is a parabolic cycle and $\Omega_i$ meets $U$, then $\Omega_i$  is the grand orbit
of  $U$ minus
the grand orbit of the critical points captured by $U$. In particular, all
component of $\Omega_i$ is iterated after after finitely many steps into $U$
with at most a countable set of points removed, and is preperiodic.
Moreover, 
$f_{|\Omega_i} : \Omega_i \rightarrow \Omega_i$ acts discretely, so
$X_i = \Omega_i/f$ is a Riemann surface isomorphic to  $X=U/f^p$ 
minus the grand orbit of critical points captured by $U$,
where $p$ is the period of the parabolic cycle associated to $U$. 

Via a Fatou coordinate, the action of $f^p$ on $U$ is conjugated to that of 
$z\mapsto z+1$ on an upper half-plane, so  $X$ is isomorphic to a cylinder
and $X_i$ is isomorphic to a cylinder with $n_i$ points removed, those points
corresponding to the $n_i$ grand critical orbits captured by $U$.
So $X$ is isomorphic to the Riemann sphere with two points $a_1$ and $a_2$ removed, and 
$X_i$ is isomorphic to the Riemann sphere with $n_i+2$ points $a_1, \ldots, a_{n_i+2}$ removed,
where the $a_j$, $j \geq 2$ correspond to the grand critical orbit meeting $U$.

By lemma \ref{lem:dicretecase}, $\dim M_f(\Omega_i)/N_f(\Omega_i) = \dim M(X_i)/N(X_i)$.
Since $X_i$ is a finitely punctured sphere, any hyperbolically bounded quasiconformal vector 
field on $X_i$ extends to a quasiconformal vector field on the torus vanishing on the 
marked points. Then the quotient $ M(X_i)/N(X_i)$ is exactly the tangent space to the 
Teichmüller space of $X_i$, which has dimension equal to the number $n_i+2-3 = n_i-1$,
where $n_i$ is the number of critical grand orbits meeting $\Omega_i$ 
(see for example \cite{hubbard2006teichmuller}). \newline

$c)$ The case of a Siegel disk

If $U$ is a Siegel disk, then the intersection of $\Lambda_f$ and the cycle of Fatou components
containing $U$ 
consists in a finite union of $n_i$ smooth circles, where $n_i$ is the number of
foliated acyclic critical classes captured by the cycle of Siegel disks  (it may be that
$n_i=0$). 
Therefore all components of $\Omega_i$ are preperiodic and are iterated in finitely many steps
to a periodic ring $A_i$ included in $U$ or a topological disk strictly included in $U$
 (if $n_i \neq 0$), or in all of the periodic Siegel disk
if $n_i=0$. In both cases, denote by  $V$ 
the periodic component of $\Omega_i$ to which is iterated a given component of $\Omega_i$.

By lemma \ref{lem:composantepreperiodique}, the space $M_f(\Omega_i)$ identifies
to the space $M_{f^p}(V)$
of Beltrami differentials on
$V$ that are invariant by $f^p_{|V}$, where $p$ is the period of the cycle associated to $U$,
and similarly $N_f(\Omega_i)$ identifies to $N_{f^p}(V)$.
A linearizing coordinate $\phi$ for $f^p$ conjugates $f^p : V \rightarrow V$ to   
$g(z)=e^{2i\pi \alpha }z$ on either the unit disk or an annulus $A(R)$, 
where $\alpha$ is an irrational rotation number. 
Therefore, by lemmas \ref{lem:invarianceparrotation} and \ref{lem:calculdimM/N}, 
$\dim M_f(\Omega_i)/N_f(\Omega_i)=1$ if $n_i \neq 0$ 
and $0$ else.   We then obtain $\sum_{j \in J} \dim M_f(\Omega_j)/N_f(\Omega_j)=n_i$.
  \newline

$d)$ The case of a Herman ring

This case is very similar to the case of a Siegel disk : $\Omega_i$ still consists
in the grand critical orbit of a periodic annulus. The only difference is that even
if there are no critical orbit lying in the Herman ring, the components
of $\Omega_i$ are still preperiodic to a ring and not a disk, and therefore
$\dim M_f(\Omega_i)/N_f(\Omega_i)=1$.  
We deduce : $\sum_{j \in J} \dim M_f(\Omega_j)/N_f(\Omega_j)=n_i+1$
where $n_i$ is the number of foliated acyclic critical classes
captured by $U$.  
 \newline

$e)$ The case of a superattracting cycle

If $U$ is a component of a superattracting cycle, 
then $\Lambda_f \cap U$ is a countable union of equipotentials (which are
smooth circles) and the superattracting cycle itself.

Assume first that there are no critical orbits captured
by the superattracting cycle. Then there is a unique $\Omega_i$
intersecting $U$, and it is the whole grand orbit of $U$. By lemma
\ref{lem:composantepreperiodique}, $M_f(\Omega_i)/N_f(\Omega_i) \simeq M_{f^p}(U)/N_{f^p}(U)$, where $p$  
is the period of $U$. 
Through a Böttcher coordinate, $f^p_{|U} : U \rightarrow U$ is conjugated to $g(z)=z^k$, $k \geq 2$.
By lemma \ref{lem:invarianceparrotation}, every Beltrami differential $\mu \in M_{f^p}(U)$
is invariant by rotation.
By lemma \ref{lem:calculdimM/N}, we deduce that if there are no critical orbits
meeting $U$, then $\dim M_{f^p}(U)/N_{f^p}(U)=0.$

Assume now that $n_i>0$, where $n_i$ is the number of foliated acyclic critical classes
meeting $U$.
Let us denote by $r_j$, $j \leq n_i$, the radii in Böttcher coordinates 
of the circles corresponding to foliated acyclic critical classes in $U$.
 Note $A(r,r')$ the annulus $\{r'<|z| <r\}$.
Let $\Omega_j \subset \Omega_f$ meeting $U$. Then for every component $V$
of $\Omega_j$, there exists a unique branch of $f^{-k} \circ f^l$ mapping $V$ into
the annulus $A(r_{j-1}, r_j)$ (with the convention $r_{-1}=1$). By lemma \ref{lem:invarianceparrotation}, 
$M_f(\Omega_j)$ identifies to
$M(A(r_{j-1,r_j}))$, and $N_f(\Omega_j)$ to $N(A(r_{j-1,r_j}))$.
We deduce from this that $\dim M_f(\Omega_i)/N_f(\Omega_i)=1$. 
Therefore $\sum_{j \in J} \dim M_f(\Omega_j)/N_f(\Omega_j)=n_i$. \newline

Summing things up, each Fatou component $U$
contributes $n_i$ to the dimension, where $n_i$ is the number of foliated acyclic critical classes
meeting $U$, except for Herman rings which contribute  $n_i+1$ and the
parabolic basins which contribute  $n_i-1$.

Moreover, ergodic line fields form a basis of the vector space $\mathrm{Fix}_J$ of invariant line fields, 
therefore $\dim \mathrm{Fix}_J = n_J$.
Thus we have : 
$$\mathrm{rg} D\Psi(0)=n_H + n_J + n_f  - n_p.$$

\end{proof}

\section{Proof of the main theorem}\label{sec:mainproof}

The first application of theorem \ref{th:calculrang} is that $\Psi^Z$ 
has constant rank :

\begin{coro}\label{coro:rgconstant}
Let $f$ be a rational map, $Z$ be an invariant set of cardinal $3$ and $\mu \in \Belf$.
Then $\mathrm{rg} D \Psi^Z(0)= \mathrm{rg} D \Psi^Z(\mu)$.
\end{coro}

\begin{proof}
It is clear that $n_f$, $n_p$ and $n_H$ are invariant under quasiconformal conjugacy. 
The number $n_J$ is invariant as well since a quasiconformal homeomorphism 
preserve sets of Lebesgues measure zero (see \cite{gardiner2000quasiconformal}). Therefore if $\phi : \rs \rightarrow \rs$
is a quasiconformal conjugacy between $f$ and another rational map 
$g$, then $\phi^*$ maps 
invariant line fields for $f$ to invariant line fields for $g$.
Lemma \ref{lem:ranguniforme} concludes the proof.
\end{proof}

\begin{coro}\label{coro:commutantssvar}
The group $\qc(f)$ is a Banach submanifold
of $\Belf$, of tangent space to the identity equal to the space $N_f(\Omega_f)$ of Beltrami
differentials of the form $\dbar \xi$, where $\xi$ is a quasiconformal vector field
invariant by $f$.
\end{coro}

\begin{proof}
The space of quasiconformal homeomorphisms commuting with $f$ is exactly the fiber
$\Psi^{-1}(f)$. But by the above corollary, $\Psi^Z$ has constant finite rank on 
$\Belf$, therefore by the constant rank theorem, $(\Psi^Z)^{-1}(f)$ is a Banach submanifold
of finite codimension, whose tangent space to the identity is $\ker D\Psi(0)=N_f(\Omega_f)$.
Moreover, $N_f(\Omega_f)$ is also the space of Beltrami differentials of the form $\dbar \xi$, 
where $\xi$ is a 
quasiconformal vector field invariant by $f$ by proposition \ref{prop:noyaudiffphi}.
\end{proof}

Note that in particular, $\mathrm{QC}(f)$ is locally connected at the identity, and therefore 
on a neighborhood of the identity, any element of $\qc(f)$ belongs also to $\qc_0(f)$.

\begin{coro}
There exists a unique structure of complex manifold on $\teich(f)$ making 
the projection $\pi : \Belf \rightarrow \teich(f)$ holomorphic. For this complex structure,
$\pi$ is a split submersion.
\end{coro}

\begin{proof}
Let $\mu \in \Belf$.
By the constant rank theorem \ref{coro:rgconstant}, there exists germs of biholomorphisms
$\phi : (\mathrm{Im} D\Psi^Z(\mu)\oplus \ker  D\Psi^Z(\mu),0) \rightarrow (\Belf,\mu)$ et $\chi : (\ratd,g) \rightarrow (\ratd,g)$
such that $\chi \circ \Psi^Z \circ \phi(u,v) = u$ for all $(u,v) \in \mathrm{Im} D\Psi^Z(\mu) \oplus \ker D\Psi^Z(\mu)$, 
where $g = \phi_\mu \circ f \circ \phi_\mu^{-1}$.

In particular, $\Psi^Z \circ \phi(u_1,v_1)=\Psi^Z \circ \phi(u_2,v_2)$ if and only if
$u_1=u_2$ ; moreover, if we note $\mu_i = \phi(u_i,v_i)$, $1\leq i \leq 2$, then
$\Psi^Z(\mu_1) = \Psi^Z(\mu_2)$ if and only if $\phi_1^Z \circ (\phi_2^Z)^{-1} \in \qc(f)$
where $\phi_i^Z$ is the quasiconformal homeomorphism corresponding to $\mu_i$ and fixing $Z$.

We claim that $\pi(\mu_1) = \pi(\mu_2)$ if and only if $u_1 = u_2$. Indeed, if
 $\pi(\mu_1)=\pi(\mu_2)$, 
then  $\phi_1^Z \circ (\phi_2^Z)^{-1} \in \qc_0(f)$ and in particular 
$\phi_1^Z \circ (\phi_2^Z)^{-1} \in \qc(f)$, so $u_1=u_2$. If now we 
assume that $u_1=u_2$, then $\psi := \phi_1^Z \circ (\phi_2^Z)^{-1} \in \qc(f)$, and we have 
to prove that in fact $\psi \in \qc_0(f)$. Let $\phi_i^Z(t)$ be the quasiconformal 
homeomorphisms corresponding to $\mu_i(t)= \phi(t u_i,t v_i)$, $1 \leq i \leq 2$ and
$t \in [0,1]$, and $\psi_t = \phi_1^Z(t) \circ (\phi_2^Z)(t)^{-1}$.
Since for all $t \in [0,1]$, $\mu_i(t) = \phi^{-1}(tu_1, tvi)$, we have 
$\psi_t \in \qc(f)$, and $\psi_0 = \mathrm{Id}$. The maps $t \mapsto \mu_i(t)$ are 
analytic, so by the parametric Ahlfors-Bers theorem so are the maps $t \mapsto \phi_i^Z(t)$.
Therefore, for all $z \in \rs$, the map 
$t \mapsto \psi_t(z)= \phi_1^Z(t) \circ (\phi_2^Z)(t)^{-1}(z)$ is continuous and $\psi_t$
is an isotopy to the identity through elements of $\qc(f)$. Moreover, 
since the $\phi_i^Z(t)$ have uniformly bounded dilatation, so does $\psi_t$. 
\footnote{Note however that the Beltrami coefficient of $\psi_t$ needs not a priori depend continuously 
on $t$.}
 So $\psi = \psi_1 \in \qc_0(f)$,
which proves the claim.

Therefore the map $\tilde{\phi} : \mathrm{Im} D\Psi^Z(0) \rightarrow \ratd$ defined by
$\tilde{\phi}(u)=\pi \circ \phi(u,0)$, where $\pi : \Belf \rightarrow \teich(f)$ is the projection, 
is a germ of homeomorphism  and makes the following diagram commute :

$$
\xymatrix{ 
  \mathrm{Im} D\Psi^Z(0) \oplus \ker D\Psi^Z(0) \ar[rr]^{\phi} \ar[d]^{\pi_1} & & \Belf \ar[d]^{\pi} \\
  \mathrm{Im} D\Psi^Z(0)  \ar[rr]^{\tilde{\phi}} & &\teich(f) \\}
$$

The map $\pi_1 : \mathrm{Im} D\Psi^Z(\mu) \oplus \ker D\Psi^Z(\mu) \rightarrow \mathrm{Im}
 D\Psi^Z(\mu)$ being the projection onto the first factor.

We can now define local sections of $\pi$ by transporting local holomorphic sections
of $\pi_1$ through the $\phi$ coordinates. 

Let us prove that these local sections of $\pi$ can be glued together compatibly to define
a complex atlas on $\teich(f)$.
Let $h_1, h_2$ be two such local sections of $\pi$ defined in a neighborhood of $[\mu] \in \teich(f)$ :
we must prove that $h_2 \circ h_1^{-1}=h_2 \circ \pi : \Belf \rightarrow \Belf$ is 
holomorphic. 
Let $g_1$ and $g_2$ be the corresponding right inverses of
$\pi_1 : \mathrm{Im} D\Psi^Z(\mu) \oplus \ker D\Psi^Z(\mu) \rightarrow \mathrm{Im}
 D\Psi^Z(\mu)$.

The following diagram commutes :

$$
\xymatrix{ 
    \Belf \ar[rr]^{h_2 \circ \pi}   \ar@<2pt>[dr]^\pi &   
     & \Belf \ar@<2pt>[dl]^\pi  \\
     &     \teich(f) \ar@<2pt>[lu]^{h_1}  \ar@<2pt>[ru]^{h_2}   
     &   \\}
$$

and therefore this diagram commutes as well :

$$
\xymatrix{ 
    \Belf \ar[rr]^{h_2 \circ \pi}  \ar[d]^{\phi^{-1}} \ar@<2pt>[dr]^\pi &   
     & \Belf \ar@<2pt>[dl]^\pi  \\
     \mathrm{Im} D\Psi^Z(\mu) \oplus \ker D\Psi^Z(\mu)   \ar@<2pt>[dr]^{\pi_1}
     &     \teich(f) \ar@<2pt>[lu]^{h_1}  \ar@<2pt>[ru]^{h_2} \ar[d]^{\tilde{\phi}}   
     &  \mathrm{Im} D\Psi^Z(\mu) \oplus \ker D\Psi^Z(\mu) \ar[u]^\phi  \ar@<2pt>[dl]^{\pi_1}  \\
     &  \mathrm{Im} D\Psi^Z(\mu)\ar@<2pt>[lu]^{g_1} \ar@<2pt>[ru]^{g_2} &  \\ }
$$

A diagramm chase then shows that $h_2 \circ \pi$ is holomorphic, 
which proves the existence of a complex structure on $\teich(f)$ meeting the requirements.

Unicity comes from the fact that for any complex structure making
the projection $\pi : \Belf \rightarrow \teich(f)$ into a holomorphic map, the local holomorphic sections of $\pi$ 
still define an atlas.

Lastly, the fact that $\pi$ admits local holomorphic sections is precisely equivalent to $\pi$ 
being a split submersion.
\end{proof}

We can finally prove the main theorem :

\begin{thmM}
The map $\Psi_T^Z : \teich(f) \rightarrow \ratd$ is an immersion,
whose image is transverse to $\of$.
\end{thmM}

\begin{proof}
By corollary \ref{coro:rgconstant} it is enough to show that it is an immersion at $0$.
By definition, we have $\Psi^Z = \Psi_T^Z \circ \pi$, and therefore
$$D\Psi^Z(0) = D\Psi_T^Z([0]) \circ D\pi(0)$$
Injectivity of $D\Psi_T^Z([0])$ is then equivalent to the property $\ker D\Psi^Z(0) = \ker D\pi(0)$.

By proposition \ref{prop:noyaudiffphi}, $\ker D\Psi^Z(0)=N_f(\Omega_f)$, and 
by corollary \ref{coro:commutantssvar}, $\ker D\pi(0)=N_f(\Omega_f)$, which concludes the proof.
\end{proof}

\begin{defi}
If $A \subset \rs$ is closed, we note $Q(A)$ the Banach space of integrable quadratic differentials
on $\rs$ and holomorphic on $\rs- A$, equipped with the
$L^1$ norm.
\end{defi}

\begin{coro}
We have the following identification :
$$T_0 \teich(f) = \belf/\{\dbar \xi, \xi = f^*\xi \}$$
$$T_0^* \teich(f)= Q(\Lambda_f)/\overline{\nabla_f Q(\Lambda_f)}.$$
\end{coro}

\begin{proof}
The first statement is a direct consequence of 
corollary \ref{coro:commutantssvar}.

Since $\teich(f)$ is a finite-dimensional manifold, it is enough to prove that
$$\left( Q(\Lambda_f)/\overline{\nabla_f Q(\Lambda_f)} \right) ^*$$
 identifies to $T_0 \teich(f)$.

By the Hahn-Banach theorem, every linear form on $Q(\Lambda_f)$ may be represented by a 
$L^\infty$ Beltrami differential. Moreover, if 
$q \in \overline{\nabla_f Q(\Lambda_f)}$, then $\int_{\rs} q \cdot \mu = 0$ for all
Beltrami differential $\mu$ invariant under $f$, and $\int_{\rs} q \cdot \mu=0$
for all quadratic differential $q \in Q(\Lambda_f)$ and all infinitesimally trivial
Beltrami differential on $\Lambda_f$, namely
any Beltrami differential of the form $\mu = \dbar \xi$, where $\xi$ is a quasiconformal
vector field on $\rs$ vanishing on $\Lambda_f$ (see \cite{gardiner2000quasiconformal}).
Therefore by theorem A, every continuous linear form on 
$ Q(\Lambda_f)/\overline{\nabla_f Q(\Lambda_f)}$ (for the quotient norm corresponding 
to the  $L^1$ norm)
may be represented by an element of $T_0 \teich(f)$, with the dual norm coinciding
with the quotient $L^\infty$ norm (it is the Teichmüller metric of
 $\teich(\Lambda_f)$, see \cite{gardiner2000quasiconformal}, 
\cite{hubbard2006teichmuller}).

This representation is unique, since if $\mu$ is a $L^\infty$ Beltrami differential
 annihilating all of $Q(\Lambda_f)$, then $\mu \in N_f(\Lambda_f)$ by 
theorem A.
\end{proof}

Note that we obtain that $Q(\Lambda_f)/\overline{\nabla_f Q(\Lambda_f)}$ has finite
dimension which is less than $2d-2$.

\bibliographystyle{alpha}
\bibliography{biblio}

\begin{thebibliography}{McM14}

\bibitem[BE09]{buff2009bifurcation}
Xavier Buff and Adam~L. Epstein.
\newblock Bifurcation measure and postcritically finite rational maps.
\newblock {\em Complex dynamics: families and friends; edited by Dierk
  Schleicher}, pages 491--512, 2009.

\bibitem[Ber67]{bers1967inequalities}
Lipman Bers.
\newblock Inequalities for finitely generated kleinian groups.
\newblock {\em Journal d'Analyse Math{\'e}matique}, 18(1):23--41, 1967.

\bibitem[Bra92]{branner1992cubic}
Bodil Branner.
\newblock {\em Cubic polynomials: turning around the connectedness locus}.
\newblock Matematisk Inst., Danmarks tekniske H{\o}jskole, 1992.

\bibitem[EM88]{mcmullen1988quasiconformal}
Clifford~J. Earle and Curtis~T. McMullen.
\newblock Quasiconformal isotopies.
\newblock {\em Holomorphic functions and moduli}, 1:143--154, 1988.

\bibitem[Eps09]{epstein2009transversality}
Adam~L. Epstein.
\newblock Transversality in holomorphic dynamics.
\newblock {\em http://homepages.warwick.ac.uk/$\sim$mases/Transversality.pdf},
  2009.

\bibitem[GL00]{gardiner2000quasiconformal}
Frederick~P. Gardiner and Nikola Lakic.
\newblock {\em Quasiconformal Teichm{\"u}ller Theory}, volume~76.
\newblock AMS Bookstore, 2000.

\bibitem[Hub06]{hubbard2006teichmuller}
John~H. Hubbard.
\newblock {\em Teichm{\"u}ller theory and applications to geometry, topology,
  and dynamics}, volume~1.
\newblock Matrix Pr, 2006.

\bibitem[Mak10]{makienko2008remarks}
Peter~M. Makienko.
\newblock Remarks on the dynamic of the ruelle operator and invariant
  differentials.
\newblock {\em Dal nevost. Mat. Zh.}, pages 180--205, 2010.

\bibitem[McM14]{mcmullen2014notes}
Curtis~T. McMullen.
\newblock Riemann surfaces, dynamics and geometry.
\newblock {\em
  http://math.harvard.edu/$\sim$ctm/home/text/class/harvard/275/rs/rs.pdf},
  2014.

\bibitem[MS98]{mcmullen1998quasiconformal}
Curtis~T. McMullen and Dennis~P. Sullivan.
\newblock Quasiconformal homeomorphisms and dynamics iii. the teichm{\"u}ller
  space of a holomorphic dynamical system.
\newblock {\em Advances in Mathematics}, 135(2):351--395, 1998.

\end{thebibliography}

\end{document}